%% file: paper-MilnorFiber-arxiv.tex
\documentclass[12pt]{amsart}  
\usepackage{amsthm}
\usepackage{amssymb}
\usepackage{amsfonts}
\usepackage{verbatim}
 \usepackage{amsmath}
 \usepackage{amscd}

\def\makeautorefname#1#2{\expandafter\def\csname#1autorefname\endcsname{#2}}
\makeautorefname{equation}{Equation}%
\makeautorefname{footnote}{footnote}%
\makeautorefname{item}{item}%
\makeautorefname{figure}{Figure}%
\makeautorefname{table}{Table}%
\makeautorefname{part}{Part}%
\makeautorefname{appendix}{Appendix}%
\makeautorefname{chapter}{Chapter}%
\makeautorefname{section}{Section}%
\makeautorefname{subsection}{Section}%
\makeautorefname{subsubsection}{Section}%
\makeautorefname{paragraph}{Paragraph}%
\makeautorefname{subparagraph}{Paragraph}%
\makeautorefname{theorem}{Theorem}%
\makeautorefname{theo}{Theorem}%
\makeautorefname{thm}{Theorem}%
\makeautorefname{addendum}{Addendum}%
\makeautorefname{addend}{Addendum}%
\makeautorefname{add}{Addendum}%
\makeautorefname{maintheorem}{Main theorem}%
\makeautorefname{mainthm}{Main theorem}%
\makeautorefname{corollary}{Corollary}%
\makeautorefname{corol}{Corollary}%
\makeautorefname{coro}{Corollary}%
\makeautorefname{cor}{Corollary}%
\makeautorefname{lemma}{Lemma}%
\makeautorefname{lemm}{Lemma}%
\makeautorefname{lem}{Lemma}%
\makeautorefname{sublemma}{Sublemma}%
\makeautorefname{sublem}{Sublemma}%
\makeautorefname{subl}{Sublemma}%
\makeautorefname{proposition}{Proposition}%
\makeautorefname{proposit}{Proposition}%
\makeautorefname{propos}{Proposition}%
\makeautorefname{propo}{Proposition}%
\makeautorefname{prop}{Proposition}%
\makeautorefname{property}{Property}
\makeautorefname{proper}{Property}
\makeautorefname{scholium}{Scholium}%
\makeautorefname{step}{Step}%
\makeautorefname{conjecture}{Conjecture}%
\makeautorefname{conject}{Conjecture}%
\makeautorefname{conj}{Conjecture}%
\makeautorefname{question}{Question}
\makeautorefname{questn}{Question}
\makeautorefname{quest}{Question}
\makeautorefname{ques}{Question}
\makeautorefname{qn}{Question}
\makeautorefname{definition}{Definition}%
\makeautorefname{defin}{Definition}%
\makeautorefname{defi}{Definition}%
\makeautorefname{def}{Definition}%
\makeautorefname{dfn}{Definition}%
\makeautorefname{notation}{Notation}
\makeautorefname{nota}{Notation}
\makeautorefname{notn}{Notation}
\makeautorefname{remark}{Remark}%
\makeautorefname{rema}{Remark}%
\makeautorefname{rem}{Remark}%
\makeautorefname{rmk}{Remark}%
\makeautorefname{rk}{Remark}%
\makeautorefname{remarks}{Remarks}%
\makeautorefname{rems}{Remarks}%
\makeautorefname{rmks}{Remarks}%
\makeautorefname{rks}{Remarks}%
\makeautorefname{example}{Example}%
\makeautorefname{examp}{Example}%
\makeautorefname{exmp}{Example}%
\makeautorefname{exam}{Example}%
\makeautorefname{exa}{Example}%
\makeautorefname{algorithm}{Algorith}%
\makeautorefname{algo}{Algorith}%
\makeautorefname{alg}{Algorith}%
\makeautorefname{axiom}{Axiom}%
\makeautorefname{axi}{Axiom}%
\makeautorefname{ax}{Axiom}%
\makeautorefname{case}{Case}%
\makeautorefname{claim}{Claim}%
\makeautorefname{clm}{Claim}%
\makeautorefname{assumption}{Assumption}%
\makeautorefname{assumpt}{Assumption}%
\makeautorefname{conclusion}{Conclusion}%
\makeautorefname{concl}{Conclusion}%
\makeautorefname{conc}{Conclusion}%
\makeautorefname{condition}{Condition}%
\makeautorefname{condit}{Condition}%
\makeautorefname{cond}{Condition}%
\makeautorefname{construction}{Construction}%
\makeautorefname{construct}{Construction}%
\makeautorefname{const}{Construction}%
\makeautorefname{cons}{Construction}%
\makeautorefname{criterion}{Criterion}%
\makeautorefname{criter}{Criterion}%
\makeautorefname{crit}{Criterion}%
\makeautorefname{exercise}{Exercise}%
\makeautorefname{exer}{Exercise}%
\makeautorefname{exe}{Exercise}%
\makeautorefname{problem}{Problem}%
\makeautorefname{problm}{Problem}%
\makeautorefname{probm}{Problem}%
\makeautorefname{prob}{Problem}%
\makeautorefname{solution}{Solution}%
\makeautorefname{soln}{Solution}%
\makeautorefname{sol}{Solution}%
\makeautorefname{summary}{Summary}%
\makeautorefname{summ}{Summary}%
\makeautorefname{sum}{Summary}%
\makeautorefname{operation}{Operation}%
\makeautorefname{oper}{Operation}%
\makeautorefname{observation}{Observation}%
\makeautorefname{observn}{Observation}%
\makeautorefname{obser}{Observation}%
\makeautorefname{obs}{Observation}%
\makeautorefname{ob}{Observation}%
\makeautorefname{convention}{Convention}%
\makeautorefname{convent}{Convention}%
\makeautorefname{conv}{Convention}%
\makeautorefname{cvn}{Convention}%
\makeautorefname{warning}{Warning}%
\makeautorefname{warn}{Warning}%
\makeautorefname{note}{Note}%
\makeautorefname{fact}{Fact}%
\newcommand{\Arr}{\mathcal{A}}
\newcommand{\pres}{\mathcal{P}}
\newcommand{\arr}{\mathcal{A}}
\newcommand{\barr}{\mathcal{B}}

\newcommand{\decone}{\mathbf{d}}
\newcommand{\darr}{{\bf d}\mathcal{A}}

\newcommand{\RR}{\mathbb{R}}
\newcommand{\CC}{\mathbb{C}}

\newcommand{\CP}{\mathbb{CP}}

\newcommand{\ZZ}{\mathbb{Z}}

\newcommand{\KK}{{\mathbb K}}

\theoremstyle{plain}

\newtheorem{theorem}{Theorem}[section]

\newtheorem{corollary}{Corollary}[section]

\newtheorem{lemma}{Lemma}[section]

\theoremstyle{definition}

\newtheorem{example}{Example}[section]

\newtheorem{notation}{Notation}[section]

\newtheorem{remark}{Remark}[section]

\makeatletter 
\let\c@theorem=\c@thm
\let\c@lem=\c@thm
\let\c@lemma=\c@thm
\let\c@prop=\c@thm
\let\c@cor=\c@thm
\let\c@cong=\c@thm
\let\c@defn=\c@thm
\let\c@remark=\c@thm
\let\c@example=\c@thm
\let\c@note=\c@thm
\let\c@nte=\c@thm
\let\c@observe=\c@thm
\let\c@notation=\c@thm
\makeatother

 \usepackage{amsmath}
 \usepackage{amscd}
 \usepackage{ifthen}

 \usepackage{tikz}
 \usetikzlibrary{arrows,chains,matrix,positioning,scopes}

 \makeatletter
 \tikzset{join/.code=\tikzset{after node path={%
 \ifx\tikzchainprevious\pgfutil@empty\else(\tikzchainprevious)%
 edge[every join]#1(\tikzchaincurrent)\fi}}}
 \makeatother
 
 \tikzset{>=stealth',every on chain/.append style={join},
          every join/.style={->}}
 
 \usepgflibrary{shapes.geometric}

\begin{document}

\title[Milnor fiber group homology]{The homology groups of the Milnor fiber associated to a central arrangement of hyperplanes in $\CC^3$ }
\author{Kristopher Williams}
\address{Department of Mathematics, Doane College, Crete, NE 68333, USA}
\email{kristopher.williams@doane.edu}
 
\subjclass[2000]{Primary 52C35, 32S55; Secondary 32Q55, 32S22, 57M05.}

\keywords{line arrangement, Milnor fiber, hyperplane arrangement}
 
  \begin{abstract} We use covering space theory and the fundamental group of complements of complexified-real line arrangements to explore the associated Milnor fiber. This work yields a combinatorially determined upper bound on the rank of the first homology group of the Milnor fiber. Under certain combinatorial conditions, we then show that one may determine the exact rank of the group and show that it is torsion free.
 \end{abstract} 
\maketitle

\section{Introduction}\label{rem:fiber}
Let $\arr$ be a an arrangement of $n$ hyperplanes in $\CC^{l+1}$. We may associate to the arrangement a reduced polynomial $Q(\arr) \in \CC\left[z_0,z_1,\dots,z_{l} \right]$, defined up to multiplication by a non-zero scalar, such that the kernel of the polynomial equals the union of the hyperplanes in $\arr$. If the arrangement is central, ie all hyperplanes contain the origin, then the polynomial defines a hypersurface singularity at the origin.

From the work of Milnor \cite{Milnor-SingularPoints-MR0239612}, we know that associated to any hypersurface singularity we may define a fibration with typical fiber called the Milnor fiber. In the case of a polynomial map $Q$ defining a central arrangement, the Milnor fiber $F$ is defined by $Q(z) = 1$ and we have a fiber bundle given by 
\begin{center}
\begin{tikzpicture}[start chain] {
    \node[on chain] {$F$} ;
    \node[on chain] {$M(\arr)$};
    \node[on chain, join={node[below]
          {\small $Q|_{M(\arr)}$}}] {$\CC^*$};
 }
\end{tikzpicture}\end{center}
where $M(\arr) = \CC^{l} \setminus Q^{-1}(0)$ is the complement of the arrangement and $Q|_{M(\arr)}$ is the restriction of the map defined by $Q$ to $M(\arr)$.

The Milnor fiber is also an $n$-fold cyclic covering space of the complement of the decone of the associated arrangement $M(\darr)$ (see Section \ref{ssec:mf-covering}). As the Milnor fiber is related to two different complements of arrangements, it is natural to ask what properties of arrangement complements extend to the Milnor fiber. It is known that the intersection lattice of the arrangement $L(\arr)$ determines the singular homology of the complement of an arrangement. However, it remains an open question as to whether $L(\arr)$ determines the homology of the Milnor fiber.

Further, the homology groups of the complement of an arrangement are known to be torsion free, but again it is an open question as to whether the homology groups of the Milnor fiber are torsion free. Cohen, Denham and Suciu \cite{CDS-torsion-MF-homology-MR1997327} have shown that torsion may arise if one works with multi-arrangements; however, the case of reduced arrangements remains open. 

In this note, we explore both the combinatorial determination of the homology as well as the question of torsion in the first homology group. We use the structure of the fundamental group of the complement of the arrangement to explore the topology of the associated Milnor fiber as a covering space. In Section \ref{sec:cover-group} we recall the presentation of the fundamental group as well as the relevant combinatorial covering space theory. Section \ref{sec:main} contains the main theorems concerning some cases where we show that the intersection lattice determines the homology groups and show that they are torsion free. In Section \ref{sec:example}, we give some examples and applications of the main theorems.


\section{Covering Spaces and Group Presentations}\label{sec:cover-group}

\subsection{Fundamental group of arrangement complements} $~$

 All presentations in this paper will be based on the Arvola-Randell presentations. For the convenience of the reader, we recall the algorithm for generating a presentation of a complexified-real arrangement in $\CC^2$ as given by Falk \cite{Falk-homotopy-types-MR1193601}. 
 
 Each line has a real defining form, therefore we may associate to the variety of the arrangement $\arr$ a graph $G(\arr)$ in the real plane. The graph consists of a vertex for each point where at least two hyperplanes intersect, and edges that lie on the hyperplanes. We note that edges may be rays. Give each edge a unique label, and let these be the generators in the presentation $\pres(\arr)$ of the fundamental group. Each intersection of lines in the arrangement locally has the form shown in Figure \ref{fig:fundgrp-local-intersetion}.

\begin{figure}[ht]
	\begin{center}
		\input{local-multiple-point-graph}
		\caption{Local depiction of an intersection of lines in $\CC^2$}
		\label{fig:fundgrp-local-intersetion}
	\end{center}
\end{figure}
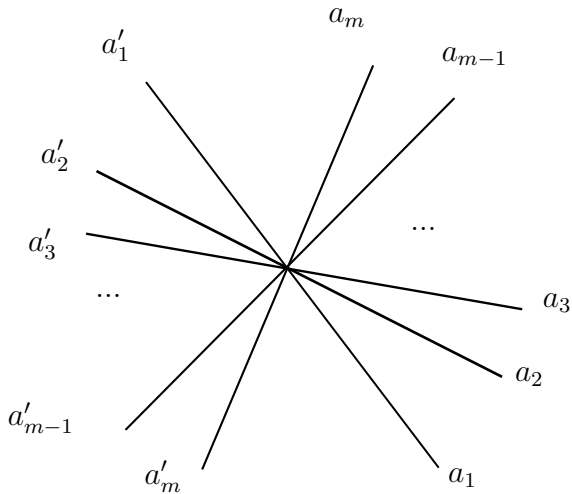

At each vertex of $G(\arr)$, we introduce the following conjugation relators,
\begin{equation*}
\begin{array}{c}
a_1'a_1^{-1} \vspace*{5pt} \\
a_2'(a_2^{a_1})^{-1} \vspace*{5pt} \\
a_3'(a_3^{a_2a_1})^{-1} \vspace*{5pt} \\
\vdots \vspace*{5pt} \\
a_{m-1}'(a_{m-1}^{a_{m-2}\cdots a_1})^{-1} \vspace*{5pt} \\
a_m'a_m^{-1} \end{array}
\end{equation*}
where $a^b = b^{-1} a b$. We also introduce the commutation relations,
\begin{equation*}
\begin{array}{c}
\left[a_m,a_{m-1}a_{m-2} \cdots a_2 a_1 \right] \\
\left[a_ma_{m-1},a_{m-2} \cdots a_2 a_1\right] \\
\vdots \\
\left[a_ma_{m-1}a_{m-2} \cdots, a_2 a_1\right] \\
\left[a_ma_{m-1}a_{m-2} \cdots a_2, a_1\right] 
\end{array}
\end{equation*}
which may be economically written as
\begin{equation*}
[a_m, a_{m-1}, \cdots, a_1].
\end{equation*}

The presentation $\pres(\arr)$ with generators described above and the relations induced by each point of intersection is called the Arvola-Randell presentation and is a presentation for $\pi_1(M(\arr))$. One may see from the conjugation relators that the number of generators in $\pres(\arr)$ can be reduced. One may reduce the number of generators by using Tietze transformations after determining the presentation. However, it is easier in practice to reduce the number of generators while constructing the presentation as follows.

 Note that we associate to each hyperplane in the arrangement two distinct rays in the graph $G(\arr)$. By choosing an arbitrary system of coordinates for $\RR^2$, we may pick a distinguished ray as follows. For each hyperplane $H \in \arr$, let $\rho_1$ and $\rho_2$ be the rays associated to the hyperplane. Let $\rho$ be the ray with the greatest $x$-coordinates, or if these are all equal, then the ray with the largest $y$-coordinate, and let $\gamma_H$ be the corresponding label. Each generator $\gamma_H$ may be identified with the homotopy class of a meridional loop around the hyperplane $H$ chosen compatibly with respect to the other generators around each hyperplane in the arrangement. By proceeding away from the ray along the line $H$, we may use the conjugation relators that involve $\gamma_H$ to express the generators coming from the other edges in terms of $\{\gamma_H\}_{H \in \arr}$.
 
 Continuing this procedure for each hyperplane in the arrangement, we arrive at a presentation with one generator for each hyperplane and $m-1$ relators for each point of multiplicity $m$ in the arrangement. 

We will denote group presentations by $\pres$ and the group associated to the presentation by $G(\pres)$.

\begin{remark}
There are several other approaches to finding presentations for fundamental groups of arrangement complements. See the work of Cohen and Suciu \cite{CS-BraidMonodromy-MR1470093} and Yoshinaga \cite{Yoshinga-lefschetz-arrang-newprent-MR2343416}.
\end{remark}

\subsection{CW-complex associated to a group presentation}\label{ssec:cw-group} $~$

Recall that the standard CW-complex associated to a (finite) presentation $P$ has one 0-cell, one 1-cell for each generator (with both ends attached to the 0-cell), and one 2-cell for each relator (with boundary attached by following along the 1-cell associated to each generator with respect to orientation). We will work with these complexes and give an example in Example \ref{ex:torus}. More details may be found in Magnus, Karass and Solitar \cite{MKS}.

\begin{example}
\label{ex:torus}
Let $\pres = \langle g_1, g_2 : 1 = g_1 g_2 g_1^{-1} g_2^{-1} \rangle$ and let $C(\pres)$ be the associated CW-complex coming from the presentation. The 1-skeleton of $C(\pres)$ is homeomorphic to $S^1 \vee S^1$. Attach a 2-cell along $g_1$ with positive orientation, $g_2$ with positive orientation, $g_1$ with negative orientation, and finally $g_2$ with negative orientation. The result is a CW-complex homeomorphic to a torus. Denote by  $C(\pres)$ the CW-complex induced by a presentation.
\end{example}
 
\subsection{The Milnor fiber as a covering space}\label{ssec:mf-covering} $~$

Let $\arr$ be a central arrangement of $n$ hyperplanes in $\CC^{l+1}$ with complement $M(\arr)$. Recall the Hopf map $h: \CC^{l+1} \setminus \{\vec{0}\} \rightarrow \CP^l$ which identifies any $z \in \CC^{l+1}\setminus \{\vec{0}\}$ with $\gamma z$ for all $\gamma \in \CC^*$. It is well known that the restriction of the Hopf map to the complement of a central arrangement is the projection map of a fiber bundle with base space the complement of the decone of the arrangement and fiber homeomorphic to $\CC^*$.

Combining the fiber bundle induced by the Hopf map with the Milnor fiber construction in Remark \ref{rem:fiber}, one may construct the commutative diagram in Figure \ref{fig:MF}.

As $f$ is homogeneous we may define the geometric monodromy $g \colon F \rightarrow F$ by $g(x_0, \dots, x_n) = (\lambda x_0, \dots, \lambda x_n)$ where $\lambda = e^\frac{2\pi i}{n}$. We note that $g$ generates a cyclic group with order $n$ acting freely on $F$. By restricting the Hopf bundle $h_F=h\vert_F \colon F \rightarrow M(\decone \arr )$, we obtain the orbit map of the free action of the geometric monodromy $g$. That is, $(h_F, M(\decone \arr))$ forms a $\ZZ_n$-bundle over $F$, hence $F/ \ZZ_n $ is homeomorphic to  $M(\decone \arr)$.

\begin{figure}[ht]
\begin{center}
\begin{tikzpicture}[scale=0.7]
 \matrix (m) [matrix of math nodes, row sep=3em,
 column sep=3em]
 { & F & F / \ZZ_n \\
 \CC^* & M(\Arr ) & M(d \Arr ) \\
 & \CC^* \\ 
 };
 \path[->] (m-1-2) edge node[above]{$p$} (m-1-3);
 \path[->] (m-1-3) edge node[right]{$\cong$} (m-2-3);

 \path[right hook->] (m-2-1) edge (m-2-2);
 \path[->>] (m-2-2) edge node[below]{$h|_{M(\Arr )}$} (m-2-3);
 \path[right hook->] (m-1-2) edge (m-2-2);
 \path[->>] (m-2-2) edge node[left] {$Q(\Arr )|_{M( \Arr )}$} (m-3-2);
 \end{tikzpicture}
\caption{Commutative diagram associated to the Milnor fiber. } \label{fig:MF}
\end{center} 
\end{figure}

The induced map $h \vert_F$ is a covering map of order $n$ induced by the map $\phi\colon \pi_1(M(\decone \arr)) \rightarrow \ZZ_n$ which sends a meridional generator to $1 \in \ZZ_n$ \cite{CS-MilnorFibers-MR1310725}, \cite{Dimca-Singularities-MR1194180}. The homotopy sequences of the fibrations fit together into the commutative diagram of Figure \ref{fig:comm-homotopy}.

\begin{figure}[ht]
\begin{center}
\begin{tikzpicture}[scale=0.7]
 \matrix (m) [matrix of math nodes, row sep=3em,
 column sep=3em]
 { & \pi_1(F) & \pi_1(F) \\
 \pi_1(\CC^*) & \pi_1(M(\Arr )) & \pi_1\left(M(d \Arr )\right) \\
 & \ZZ & \ZZ_n \\ 
 };
 \path[->] (m-1-3) edge node[right]{$p_{\#}$} (m-2-3);
 \path[->] (m-2-3) edge node[right]{$\phi$} (m-3-3);
 \draw[-] (m-1-2) edge[double=white,double distance = 3pt] (m-1-3);
 \path[->] (m-2-1) edge (m-2-2);
 \path[->] (m-2-2) edge node[below]{$(h|_{M(\Arr )})_{\#}$} (m-2-3);
 \path[->] (m-1-2) edge (m-2-2);
 \path[->] (m-2-2) edge node[left] {$(Q(\Arr )|_{M( \Arr )})_{\#}$} (m-3-2);
 \end{tikzpicture}
\caption{Homotopy groups associated to the Milnor fiber. } \label{fig:comm-homotopy}\end{center} \end{figure}

Cohen and Suciu used this construction to explore the eigenspaces of $H^*(F; \CC)$ induced by the monodromy map $q\colon F \to F$, $q(z) = \exp(2\pi i / (n+1)) \cdot z$. One result we will use is the following:

\begin{theorem}{\rm \cite[Remark 1.7]{CS-MilnorFibers-MR1310725}}
\label{thm:minimal-betti-mf} Let $\arr$ be a central arrangement of $n$ hyperplanes in $\CC^{l+1}$. If $F$ is the Milnor fiber associated to $\arr$, then the first betti number of $F$ is at least $n$.
\end{theorem}

We now restrict our attention to the case of arrangements in $\CC^3$. Consider the exact sequence $$0 \to \pi_1(F) \to \pi_1(M(\darr)) \to_\phi \ZZ_n \to 0$$
From the work of Section \ref{ssec:cw-group}, we construct a CW-complex $C$ from the presentation of $\pi_1(M(\darr))$. 

Next, we construct a CW-complex $K$ such that $\pi_1(K) \cong \pi_1(F)$. Lyndon and Schupp describe the process   explicitly in \cite{Lyndon-Schupp-Comb-Group-Theory-MR1812024} Proposition 3.4, p. 119 and we recall it in Example \ref{ex:covering}. Finally, we determine $H_1(K,\ZZ)$ and conclude that $H_1(K,\ZZ)$ is isomorphic to $H_1(F,\ZZ)$. The following lemma follows easily from the preceding statements.

\begin{lemma}\label{lem:construction}
Using the presentation for $\pi_1(M(\arr))$ and the map $\phi \colon \pi_1(M(\arr)) \to \ZZ_n$ we may construct a CW-complex $K$ such that $H_1(K, \ZZ) \cong H_1(F,\ZZ)$ where $F $ is the Milnor fiber associated to the arrangement $\arr$.
\end{lemma}

\begin{example}\label{ex:covering}
Let $\arr$ be the central arrangement in $\CC^3$ defined by $Q(x,y,z)=xyz$. Then $\pi_1(M(\darr))$ has an Arvola-Randell presentation given by 
\begin{equation}
\pres : = \langle  g_1, g_2 : 1 = g_1 g_2 g_1^{-1} g_2^{-1} \rangle,
\end{equation}
where each $g_i$ corresponds to a meridional loop around a hyperplane in $M(\darr)$. The map $\phi \colon \pi_1(M(\darr)) \to \ZZ_3$ induces the covering map associated to the Milnor fiber $F$. For ease of notation, we will use the cyclic group of order three generated by $x$, $\langle x \vert x^3=1\rangle$.

We construct the CW-complex $K$ as follows. Start by denoting three 0-cells by $\{v, xv, x^2v\}$. We then attach three 1-cells to these 0-cells for each generator in $\pres$. Denote these by $\{x^j g_1, x^j g_2\}_{j=0}^2$. Orient each 1-cell $x^jg_i$ in the positive direction from endpoint $x^j v$ to endpoint $x^{j+1}$. Attach three 2-cells, one beginning at each 0-cell. The first has oriented boundary $g_1 + xg_2 - xg_1 - g_2$; the rest follow the same pattern starting at the next vertex. 
In this case, $K$ is homotopy equivalent to the 3-fold cyclic cover of a torus. Therefore, $H_1(F,\ZZ) \cong H_1(K,\ZZ) \cong \ZZ^3$.
\end{example}

\begin{notation}\label{not:tuple}
Rather than writing each of the boundaries as a sum, we use tuple notation. More explicitly, the tuple \begin{equation*}
(g_1- g_2, g_2 - g_1,0)
\end{equation*}
will be the notation for the oriented boundary $g_1 + xg_2 - xg_1 - g_2$. The index in the tuple is one more than the exponent on the ``$x$'' coefficient of the lift. This allows us to define an action of $\ZZ_n = \langle x: x^n=1 \rangle$ on the tuples. The coefficient on $x$ indicates the number of positions to move right in the tuple, and whenever an element runs out of positions in the tuple, it moves to the beginning of the tuple. For example, $x^2.(a,b,c,0)$ is equal to $(c,0,a,b)$ which may be written as the sum $c + x^2a + x^3b$.
\end{notation}

\section{Main Theorems}\label{sec:main}

\subsection{Single Point of Intersection}$~$

\label{ssec:one-point-intersection}
In this section, we closely examine an arrangement that consists of a pencil of lines, ie all lines intersect in a single point. This local calculation forms the basis for many of the arguments presented in the subsequent sections.

Explicitly, let $\darr$ be the arrangement defined by the polynomial $Q(\darr)= \left(y-x\right)\left(y-2x\right) \cdots \left(y-(n-1)x\right)$. Let $g_k$ be the homotopy class of a meridional generator around the line defined by $\left(y-kx\right)$. The Arvola-Randell presentation for $\pi_1(M(\darr))$ is given by 
\begin{equation}
\label{eqn:pres-pencil}
\mathcal{P} = \langle g_1, \cdots, g_n \colon [g_{1},\cdots,g_{n-1}] \rangle
\end{equation}

In Section \ref{ssec:mf-covering}, we defined the map $\phi \colon \pi_1(M(\darr)) \to \ZZ_n$ induced by sending $g_k$ to $1 \in \ZZ_n$. We will now consider a more general situation. Let $\phi_m \colon \pi_1(M(\darr)) \to \ZZ_m$ be the map induced by sending $g_k$ to $1 \in \ZZ_m$. Using this map, one may construct a CW-complex $K(\mathcal{P},\phi_m)$ that has the homotopy type of an $m$-fold cyclic covering of $C(\mathcal{P})$, the CW-complex induced from the presentation $\mathcal{P}$.

We begin by explicitly examining $C(\mathcal{P})$. The relations given by $[g_1,\dots,g_n]$ are
\begin{equation*}
g_{n-1} \cdots g_2 g_1 = g_{n-2}g_{n-3} \cdots g_2 g_1 g_{n-1} = \cdots = g_1 g_{n-1} \cdots g_2
\end{equation*}

We are able to rewrite these relations as the following relators:
\begin{align*}
R_1 &:=g_{n-1}  g_{n-2} \cdots g_2 g_1 g_{n-1} ^{-1} g_1^{-1} g_2 ^{-1} \dots g_{n-2}^{-1} \\
R_2 &:=g_{n-2}g_{n-3} \cdots g_2 g_1 g_{n-1}  g_{n-2}^{-1} g_{n-1} ^{-1} g_1^{-1} \dots g_{n-3}^{-1} \\
\cdots &\cdots \\
R_{n-2} &:=g_2 g_1 g_{n-1} \cdots g_3 g_2^{-1} g_3^{-1} \dots g_{n-1} ^{-1}g_1^{-1}
\end{align*}

These are the boundaries of 2-cells in $C(\pres)$, thus we lift them to boundaries of 2-cells in $K(\pres,\phi_m)$ (using tuple notation, see Notation \ref{not:tuple}) as
\begin{align*}
x^i .\Tilde{R_1} &:= x^i . (g_{n-1} - g_{n-2}, g_{n-2} - g_{n-3}, \dots, g_2 - g_{1}, g_{1} - g_{n-1}, 0,\dots,0) \\
x^i .\Tilde{R_2} &:= x^i . (g_{n-2} - g_{n-3}, g_{n-3} - g_{n-4}, \dots, g_1 - g_{n-1}, g_{n-1} - g_{n-2}, 0,\dots,0) \\
\cdots& \cdots \\
x^i .\Tilde{R}_{n-2} &:= x^i . (g_2 - g_{1}, g_{1} - g_{n-1}, \dots, g_3 - g_{2}, 0,\dots,0)
\end{align*}
for $0 \leq i \leq m-1$ and denote this set by $\Tilde{R}$. Note that each tuple has $m$ positions.

  We use parenthesis to denote abelian presentations and can see that 
\begin{equation*}
H_1(K(\pres,\phi_m), \ZZ) = \left( \begin{array}{l|r}
	 (g_1,\dots,g_1) &  \Tilde{R} \\
	 x^i(g_{j+1} - g_j,0,\dots,0), 1\leq j \leq n-2, 0 \leq i \leq m-1 ~& 
	\end{array} \right)
\end{equation*}

We will use Tietze transformations to determine a more useful presentation of this group. We have $\{x^i. \Tilde{R}_j\}_{i,j}$ as the generating set for the relators. We now replace $x^i. \Tilde{R}_1$ by $$x^i. \Tilde{R}_1^* = x^i .\Tilde{R}_1 - x^{i+1} .\Tilde{R}_2$$ for $0 \leq i \leq m-1$. Continue replacing $x^i .\Tilde{R}_j$ by $$x^i. \Tilde{R}_j^* = x^i. \Tilde{R}_j - x^{i+1}. \Tilde{R}_{j+1}$$ for $0 \leq i \leq m-1$ and $1 \leq j \leq n-3$.

Thus we end up with
\begin{align*}
 x^i .\Tilde{R}_1^* &:= x^i .( g_{n-1}-g_{n-2},0,\dots,0, - g_{n-1} + g_{n-2},0\dots,0 ) \\
 x^i .\Tilde{R}_2^*&:= x^i .( g_{n-2}-g_{n-3},0,\dots,0, - g_{n-2} + g_{n-3},0\dots,0 ) \\
 \cdots & \cdots \\
 x^i .\Tilde{R}_{n-3}^*&:=x^i .( g_{3}-g_{2},0,\dots,0, - g_{3} + g_{2},0\dots,0 ) \\
x^i . \Tilde{R}_{n-2} &:= x^i . (g_2 - g_{1}, g_{1} - g_{n-1}, \dots, g_3 - g_{2}, 0 , \dots, 0)
\end{align*}
with second non-zero entry in the $n$-th place in the tuple.

Let us now inspect the set of relators $x^i. \Tilde{R}_1^*$. Expanding with respect to $i$, we may see that this set of relators is equivalent (by simple linear algebra) to
\begin{equation*}
x^i. (g_{n-1} - g_{n-2}, 0 , \dots, 0, -(g_{n-1} - g_{n-2}), 0, \dots, 0)
\end{equation*}
where the second non-zero entry is in the $(w+1)$-th position of the tuple, where 
\begin{equation*}
w = \begin{cases}
1 & \text{ if } m=2 \\
\gcd(m,n-1) & \text{ otherwise}
\end{cases} .
\end{equation*} Let us denote this new set of expressions by $x^i \Tilde{R}_1^{**}$. We have a similar result for each set of expressions $x^i \Tilde{R}_j^*$, so call these new relators $x^i \Tilde{R}_j^{**}$. 

To summarize, we have the set of relators
\begin{equation}
\label{eqn:gcd-relators}
 \left\lbrace
\begin{array}{rl|r}
x^i. \Tilde{R}_j^{**} &:= x^i .( g_{n-j}-g_{n-j-1},0,\dots,0, - (g_{n-j} - g_{n-j-1}),0,\dots,0,) & 0 \leq i \leq m-1\\
x^i . \Tilde{R}_{n-2} &:= x^i . (g_2 - g_{1}, g_{1} - g_{n-1}, \dots, g_3 - g_{2}, 0 , \dots, 0) & 1 \leq j \leq n-3
\end{array}
 \right\rbrace
\end{equation}
where the second non-zero entry is in the position described in the previous paragraph.

\begin{remark}
\label{rem:rel-prime-relators-reduced}
Suppose now that $\gcd(m,n-1)=w=1$. This is exactly the case for a pencil of lines. We now have the set of relators
\begin{equation}
\label{eqn:rel-prime-relators}
\left\lbrace
\begin{array}{rl|r}
x^i.\Tilde{R}_j^{**}&:= x^i . (g_{n-j}-g_{n-j-1}, -(g_{n-j}-g_{n-j-1}), 0 \dots, 0) & 0 \leq i \leq m-1\\
x^i . \Tilde{R}_{n-2}&:= x^i . (g_2 - g_{1}, g_{1} - g_{n-1}, \dots, g_3 - g_{2},0,\dots,0) & 1 \leq j \leq n-3 
\end{array}
 \right\rbrace
\end{equation}

Let us now examine $x^i. \Tilde{R}_{n-2}$. We rewrite the set as
\begin{equation*}
x^i . \left(g_2 - g_{1}, \sum_{k=2}^{n-1} (-1)(g_{k} - g_{k-1}),g_{n-1} - g_{n-2} \dots, g_3 - g_{2}, ,0,\dots,0\right)
\end{equation*}
By using the relators $x^i . \Tilde{R}_j^{**}$, we have cancellation of most terms and have the new relator $x^i . \Tilde{R}_{n-1}^{**}$:
\begin{equation*}
x^i . (g_2 - g_1, g_1 - g_2, 0 , \dots, 0)
\end{equation*}
Therefore we conclude that $\Tilde{R}$ is equivalent to the set of relators below.
\begin{equation}
\label{eqn:rel-prime-relators-reduced}
\left\lbrace
\begin{array}{l|r}
 x^i. (g_j - g_{j-1}, -( g_j - g_{j-1}), 0, \dots, 0) & 0 \leq i \leq m-1\\
 & 2 \leq j \leq n-1
 . \end{array}
 \right\rbrace
 \end{equation}
 
The presentation for $H_1(K(\pres, \phi_m), \ZZ)$ is
\begin{equation*}
H_1(K(\pres,\phi_m), \ZZ) = \left( \begin{array}{l|r}
	 (g_1,\dots,g_1) &   \Tilde{R} \\
	 x^i(g_{j+1} - g_j,0,\dots,0), 1\leq j \leq n-2, 0 \leq i \leq m-1 ~& 
	\end{array} \right) 
\end{equation*}

\begin{equation*}
= \left( \begin{array}{l|r}
	 (g_1,\dots,g_1) &   \\
	 (g_{j+1} - g_j,0,\dots,0), 1\leq j \leq n-2  ~& 
	\end{array} \right) 	
	\end{equation*}
\begin{equation*}
\cong \ZZ^{n-1}
\end{equation*}
\end{remark}

\subsection{Arrangements intersecting in general position}$~$

We begin by recalling the following theorem of Oka and Sakamoto \cite{Oka-Sakamoto-ProductTheorem-MR513072} concerning algebraic plane curves:

\begin{theorem}{\rm \cite{Oka-Sakamoto-ProductTheorem-MR513072}}\label{thm:os-split} Let $C$ be an algebraic plane curve in $\CC^2$ such that $C = C_1 \cup C_2$, where $C_i$ has degree $d_i$. If $C_1 \cap C_2$ consists of $d_1 \cdot d_2$ points, then 
\begin{equation*}
\pi_1(\CC^2 - (C_1 \cap C_2)) \cong \pi_1(\CC^2 - C_1) \times \pi_1(\CC^2 - C_2).
\end{equation*} 
\end{theorem}

The importance of this theorem comes from its proof. In order to establish the isomorphism, Oka and Sakamoto use a 1-parameter family of curves. More explicitly: let $f(x,y)$ and $g(x,y)$ be defining polynomials of $C_1$ and $C_2$ respectively. Letting $C_1(t(s)) = f(t(s)x,y), C_2(\tau(s))=g(x,\tau(s)y)$, the authors construct a smooth one parameter family of curves $\{C_1(t(s)) \cup C_2(\tau(s)) ; 0 \leq s\leq 1\}$. This construction is made such that $\CC^2 \setminus C_1 \cup C_2$ is homeomorphic to $\CC^2 \setminus C_1(t(s)) \cup C_2(\tau(s))$ for all $s$ and $C_1 = C_1(t(0))$, $C_2 = C_2(\tau(0))$.

They then proceed to construct a presentation (which we choose to be an Arvola-Randell presentation) for $ \pi_1(\CC^2 \setminus C_1(t(s_0)) \cup C_2(\tau(s_0))) $ given by
\begin{equation}
\label{eqn:os-pi1}
\pres = \langle a_j , b_k : [a_j,b_k], R_a , R_b \rangle
\end{equation}
where $R_a$ (respectively $R_b$) consists of relations involving only $a_i$'s (respectively $b_k$'s), and $1\leq j \leq d_1$  ($1 \leq k \leq d_2$). Additionally, the presentations for $ \pi_1(\CC^2 \setminus C_1)$ and $ \pi_1(\CC^2 \setminus C_2)$ are given respectively by
 \begin{align*}
\pres_1 &= \langle a_j : R_a \rangle \\
\pres_2 &= \langle b_k : R_b \rangle
 \end{align*} In the case when $C_1$ and $C_2$ define arrangements, the presentations given above may chosen to be Arvola-Randell presentations.

Using Oka and Sakamoto's presentation of the fundamental group applied to arrangement complements, we may prove the following theorem.
\begin{theorem}
\label{thm:OS-MF}
 Let $\mathcal{C}$ be an arrangement of $n$ lines in $\CC^2$ such that $\mathcal{C} = \mathcal{A} \coprod \mathcal{B}$ such that $\mathcal{A}$ and $\mathcal{B}$ intersect in general position, ie $\mathcal{C}$ satisfies the hypotheses of Theorem \ref{thm:os-split}. If $F$ is the Milnor fiber associated to the cone over this arrangement, then $H_1(F) \cong \ZZ^{n}$. 
\end{theorem}

\begin{proof}
Assume that the arrangement decomposes into two subarrangements $\arr$ and $\barr$ of degrees $d_1$ and $d_2$ respectively. Then we may conclude by Theorem \ref{thm:os-split} and \eqref{eqn:os-pi1} that
\begin{equation*}
\pres := \langle a_j , b_k : [a_j,b_k], R_a , R_b \rangle
\end{equation*}
where the $a_j$'s and $b_k$'s are generators corresponding to transverse loops around hyperplanes in $\arr$ and $\barr$ respectively and $\pres$ is an Arvola-Randell presentation for $\pi_1(M(\arr))$.

Using $\pres$  and $\phi\colon G(\pres) \twoheadrightarrow \ZZ_{n+1}$ (the homomorphism induced by $a_j \mapsto 1$, $b_k \mapsto 1$), by Lemma \ref{lem:construction} we have a CW-complex $K$ such that $H_1(K,\ZZ) \cong H_1(F,\ZZ)$.

We have the following boundaries in $K$ coming from lifts of $[a_j,b_k]$:
\begin{equation}
\label{eqn:relators-commute}
\left\lbrace
\begin{array}{l|r}
&~~0 \leq t \leq n \\
x^t .(a_j -b_k, - a_j + b_k , 0, \dots, 0) &~~1 \leq j \leq d_1 \\
&~~1 \leq k \leq d_2
 \end{array}
\right\rbrace.
\end{equation}

We now collapse the CW-complex $K$ along the maximal tree given by $\{xa_1,x^2 a_1, \dots, x^{n} a_1\}$. This allows us to give a presentation 
 \begin{equation}
\label{eqn:homology-split}
\begin{array}{lr}
H_1(K) &= \left( (a_1,0,\dots,0), x^t.( a_j, 0, \dots, 0) , x^t.( b_k,0,\dots,0) : \Tilde{R} \right) \\
 &~~~~~~~0 \leq t \leq n, ~2 \leq j \leq d_1 \text{, and}~ 1\leq k \leq d_2 
\end{array}
\end{equation}
where the relators in $\Tilde{R}$ coming from \eqref{eqn:relators-commute} are 

\begin{equation}
\label{eqn:fromR-1}
\left\lbrace
\begin{array}{l|r}
(a_1 -b_k, b_k, 0, \dots, 0)& 1\leq k \leq d_2\\
x^\tau (-b_k, b_k, 0, \dots, 0) & 1 \leq \tau \leq n \\
(b_k - a_1,0, \dots, 0, - b_k)& 
\end{array}
\right\rbrace
\end{equation}
and
\begin{equation}
\label{eqn:fromR-2a}
\left\lbrace
\begin{array}{l|r}
& 2 \leq j \leq d_1 \\
x^t. (a_j - b_k, b_k - a_j, 0, \dots, 0) & 0 \leq t \leq n \\
& 1\leq k \leq d_2
\end{array}
\right\rbrace
\end{equation}

Using the relators in \eqref{eqn:fromR-1}, we may rewrite the relators in \eqref{eqn:fromR-2a} as
\begin{equation}
\label{eqn:fromR-2b}
\left\lbrace
\begin{array}{l|r}
(a_1 - a_j, a_j, 0, \dots, 0) & 2 \leq j \leq d_1 \\
x^\tau . (-a_j , a_j, 0, \dots, 0) & 0 \leq t \leq n \\
(a_j - a_1, 0, \dots, 0 - a_j) & 
\end{array}
\right\rbrace
\end{equation}

Therefore, by using the relators in \eqref{eqn:fromR-1} and \eqref{eqn:fromR-2b} on the basis given for $H_1(K)$ in \eqref{eqn:homology-split}, we have that the number of generators for $H_1(K)\cong H_1(F)$ is at most $d_1 + d_2$. However, by Theorem \ref{thm:minimal-betti-mf}, we have that $b_1(F) \geq b_1(M(\arr)) = d_1 + d_2 = n$. Thus we have $H_1(F) \cong \ZZ^n$.
\end{proof}

\subsection{Combinatorial criterion for an upper bound} $~$

We now prove the main theorem of this paper:

\begin{theorem}
\label{thm:OneHyp}
 Let $\arr$ be a complexified real-arrangement of $n$ hyperplanes in $\CP^2$, and let $F$ be the associated Milnor fiber. For any hyperplane $H \in \arr$, let $V$ be the set of multiple points on $H$, and let $m_v$ denote the multiplicity of the point $v \in V$. Then, for any field $\KK$, the first betti number of $F$ with respect to $\KK$ satisfies 
 \begin{equation*}
b_1(F,\mathbb{K}) \leq (n-1) + \sum_{v \in V} \left[ (m_v -2)(\gcd(m_v,n) -1) \right]
\end{equation*}
 \end{theorem}

  As a corollary to the Theorem, we have the following:
  \begin{corollary}
  	\label{Combo}
  	Let $\arr$ be a complexified real-arrangement of $n$ hyperplanes in $\CP^2$, and let $F$ be the associated Milnor fiber. 
  	If there exists $H \in \arr$ such that for all $v \in V_{H}$ we have $m(v) =2$ or $\gcd(m(v),n) = 1$, then $H_1 (F; \ZZ) \cong \ZZ^{n-1}.$
  \end{corollary}
  \begin{proof}The statement follows immediately from Theorem \ref{thm:OneHyp} as we have 
  	\begin{align*}
  	b_1(F,\KK) &\leq (n-1) + \sum_{v \in V} \left[ (m_v -2)(\gcd(m_v,n) -1) \right]\\
  	& = (n-1) + 0 
  	\end{align*}
  	
  	From Theorem \ref{thm:minimal-betti-mf} we have that $b_1(F,\CC) \geq b_1(\decone \arr) = n-1$. Any $p$-torsion would be detected by an increase in the betti number with respect to a field of characteristic $p$. Since the numbers have the same bound for all fields, we may conclude that $H_1(F, \ZZ)$ $\cong \ZZ^{n-1}$.\end{proof}

\begin{remark}Using perverse sheaves and vanishing cycles, Massey \cite{Massey-Perversity-duality-MR1416758} showed the same upper bounds as Theorem \ref{thm:OneHyp} for the betti number with complex coefficients. We are not sure if his methods may be used to prove the same upper bound holds for betti numbers with coefficients in fields of all characteristics.
\end{remark}

In order to prove Theorem \ref{thm:OneHyp}, we use the presentation for the fundamental group induced by considering the arrangement as an arrangement of projective lines in the complex projective plane. Let $\arr = \{H_i\}_{i=0}^n$ be a central arrangement of hyperplanes in $\CC^{l+1}$. Let $h\colon \CC^{l+1} \setminus \{0\} \to \CP^l$ denote the projection map of the Hopf bundle with fiber $\CC^*$. Then $h$ maps the complement of the arrangement $M(\arr)$ onto the complement of the projectivization of the arrangement, which we identify with the complement of the decone of the arrangement $M(\darr)$ in $\CC^n$ (see \cite{OT-Arrs-MR1217488}). 

Meridional loops $g_{H_i}$, one around each hyperplane in $\arr$, generate the fundamental group $\pi_1(M(\arr))$. These may be chosen compatibly so that $\prod g_{H_i} $ is null homotopic in $\pi_1(M(\darr))$. In fact, $\pi_1(M(\arr)) / \langle \prod g_{H_i}  \rangle \cong \pi_1(M(\darr))$ (Section 3.1 \cite{CDS-torsion-MF-homology-MR1997327}, \cite{Garber-affine-proj-MR2145948}).

The advantage to this approach is that in the generic section used to calculate $\pi_1(M(\arr))$ any two distinct hyperplanes intersect in one point. However, we are left with more relations in the resulting presentation for $\pi_1(M(\darr))$ than if we had calculated $\pi_1(M(\darr))$ directly via the Arvola-Randell algorithm.

\begin{example} Let $\arr$ be the arrangement in $\CC^3$ with defining polynomial  $Q(\arr)=xyz$. The complement of the decone of this arrangement may be identified with the complement of $\arr$ considered as an arrangement in the complex projective plane. We then have a presentation for $\pi_1(M(\darr))$ given by $$\langle g_x, g_y, g_z : [g_x,g_y], [g_x,g_z],[g_y,g_z] , g_x g_y g_z =1 \rangle.$$ This presentation is equivalent to the presentation for $\pi_1(M(\darr))$ considered as an arrangement in $\CC^2$: $$\langle h_x, h_y : [h_x,h_y] \rangle.$$\end{example}

From Section \ref{ssec:mf-covering} we see that we may use any presentation for the arrangement as long as we may relate the generators to the associated map $\phi \colon \pi_1(M(\darr)) \to \ZZ_n$. Let $g_{H_n}$ be the meridional generator associated to the hyperplane at infinity. Then as we know that $\phi(g_{H_i}) = 1$ for $0 < i < n-1$, one may see that the relator $\prod g_{H_i}$ forces $\phi(g_{H_n}) = 1$ as well. Therefore, we extend the map $\phi$ to be a map on the presentation induced by considering the presentation in projective space and keep the same notation.

\subsection{Proof of Theorem \ref{thm:OneHyp}} $~$
 
As $\arr$ is a complexified-real arrangement in $\CP^2$, we may depict the arrangement along an arbitrary hyperplane $H$ as in Figure \ref{fig:one-hyp-cp2}. We label the other hyperplanes according to their intersections with $H$ in $k= |V|$ distinct points by two indices: the first indicating the intersection point, the second the position of the hyperplane with respect to the intersection point. Note that the multiplicity of each intersection point $p$ is given by $m_p$, and that $H$ is always the first hyperplane with respect to each intersection point. We then have
\begin{equation*}
\arr = \{ H, A_{1,2}, A_{1,3}, \dots, A_{1,m_1}, A_{2,2}, \dots, A_{k,m_k} \}.
\end{equation*}
\begin{figure}[ht!]
	\begin{center}
	\input{one-line-intersection}
	\caption{Depiction along $H$ of the arrangement.}

\label{fig:one-hyp-cp2}
\end{center}
\end{figure}
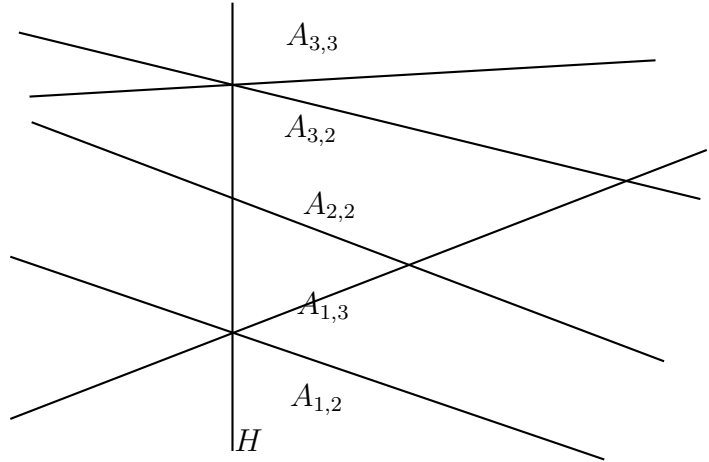
We denote the meridional loop coming from each hyperplane by a corresponding lower case letter and have the following presentation from the Arvola-Randell algorithm 

\begin{equation*}
\pres = \langle h,a_{1,2},a_{1,3}, \dots, a_{1,m_1}, a_{2,2}, \dots, a_{k,m_k} : R_H \cup P \cup R \rangle
\end{equation*}
where

\begin{equation*}
R_H = \left\{ \begin{matrix}
[h,a_{1,2},\dots, a_{1,m_1}]\\ 
[h,a_{2,2},\dots, a_{2,m_2}]\\ 
[h,a_{3,2 },\dots, a_{3,m_3}] \\ 
\dots \\
[h,a_{k,2 },\dots, a_{k,m_k}] 
\end{matrix} \right\}
\end{equation*}

\begin{equation*}
P = \{h\cdot  \prod_{j=2}^{m_1} a_{1,j} \cdots \prod_{j=2}^{m_k} a_{k,j}=1\},
\end{equation*}
where the relations in $R_H$ are the relations coming from multiple points along the hyperplane $H$, $R$ is used to denote the other relations, and $P$ is the projective relation.

 As before, define $\phi:\pi_1(M(\arr)) \rightarrow \ZZ_n$ by sending all meridional generators to $1 \in \ZZ_n$ and let $K = K(\pres,\phi)$. 
\begin{lemma}\label{lem:gens} 
$H_1(K)$ has a presentation as $(G:\Tilde{R} \cup \Tilde{P} \cup \Tilde{R}_H)$ with 
\begin{equation}
\label{eqn:gens}
G=\left\lbrace
\begin{array}{l|r}
x^j .(a_{i,2} - h,0,\dots, 0) & 0 \leq j \leq n-1 \\
x^j.(a_{i,p_i+1} - a_{i,p_i},0, \dots, 0) & 1 \leq i \leq k \\
(h, a_{1,2}, \dots, a_{k,m_k}) & 2 \leq p_i < m_i
\end{array}
\right\rbrace
\end{equation}
\end{lemma}
\begin{proof} 
From the work in Section \ref{ssec:mf-covering}, we construct the CW-complex $K=K(\pres, \phi)$ such that the first homology of the Milnor fiber is isomorphic to that of $K$. Next, we collapse the 1-skeleton of $K$ along edges of the form $x^i h$ for $0\leq i < n-1$. This is a strong deformation retract of the 1-skeleton to a wedge of 1-spheres. Therefore, each remaining edge represents a 1-cycle. By uncollapsing the edges, we may represent every 1-cycle as 
\begin{equation*}\label{eqn:gens1}
x^j.(a_{i,k} - h, 0, \dots, 0) \text{ or } (h, h, \dots, h)
\end{equation*}
The set of generators in \eqref{eqn:gens} is given by linear combinations of those given in \eqref{eqn:gens1}.
\end{proof}

Now that we have a suitable set of generators, we examine the relations. We drop the second index label for notational convenience.

\begin{lemma}\label{lem:rels}
The relations of the form $[a_1,a_2,\dots, a_m]$ may be used to reduce the number of generators of $H_1(K)$
$(n-1) + (m-2)(n - \gcd(m,n))$.
\end{lemma}
\begin{proof}
In  Section \ref{ssec:one-point-intersection} we showed that from a relation of the form given, we have relators in the covering space given by 
\begin{equation*}
x^j .(a_i- a_{i-1},0,\dots, 0,a_i - a_{i-1}, 0, \dots, 0)
\end{equation*}
with the second non-zero entry in the $(w+1)$-th position where $w=\gcd(m,n)$ and
\begin{equation*}
 x^j .( a_2 - a_1, a_1 - a_m, a_m - a_{m-1},\dots, a_3 - a_2, 0, \dots, 0)
\end{equation*}
for $3 \leq i \leq m$, and $0 \leq j < n$.
 
 Using the first set of relators, for each $i$ we may reduce our generating set by $\left(\frac{n}{w} -1 \right) w$ elements by identifying them with a generator of the form $(0, \dots, a_i - a_{i-1}, 0, \dots, 0)$ where the cycle is in the $(j+1)$-th position for $0 \leq j < w$. We repeat this for each of the $(m-2)$ relators as we have $$x^j (a_i - a_{i+1}) = x^{j+w} (a_i - a_{i+1})=x^{j+2w} (a_i - a_{i+1}) = \cdots = x^{j+sw} (a_i - a_{i+1})$$
 for all $s \in \ZZ$.
 
The second set of relators may be rewritten as
\begin{align*}
&x^j. ((a_2 - a_1), - (a_m - a_1), a_m - a_{m-1},\dots, a_3 - a_2, 0, \dots, 0) \\
&~= x^j.((a_2 - a_1), - \sum_{p=2}^{m}(a_p - a_{p-1}), a_m - a_{m-1},\dots, a_3 - a_2, 0, \dots, 0)
\end{align*}
Thus we may write any cycle $x^j.(a_2 - a_1, 0, \dots, 0)$ ($j > 0)$ in terms of other cycles, except for $(a_2 - a_1, 0, \dots, 0)$. Hence we get another elimination of $n-1$ distinct generators. (We perform this reduction by noting that all $a_2-a_1$ cycles outside the first slot may be written in terms of other cycles. One should note that the cycles are not identified ``cyclically'' as before.) Combining these with the $(m-2)(n-\gcd(m,n))$ generators we eliminated earlier, we have our lemma.\end{proof}

\begin{remark}
From the proof of Lemma \ref{lem:rels} and the group presentation, it is clear that the relators in $\Tilde{R}_H$ coming from each multiple point $v \in V$ identify a disjoint set of generators. Thus each multiple point reduces the number of distinct homology classes by $(n-1) + (m_v - 2)(n-\gcd(m_v,n))$. 
\end{remark}

\begin{lemma}
\label{lem:infinity}
The homology class generated by $(h, a_{1,2}, \dots, a_{k,m_k})$ is trivial.
\end{lemma}
\begin{proof}
Recall that the relator $P$ equals $h \cdot \prod_{j=2}^{m_1} a_{1,j} \cdots \prod_{j=2}^{m_k} a_{k,j}$. In the first homology group of $K$, this relator takes the form
\begin{equation*}
(h, a_{1,2}, \dots, a_{k,m_k})
\end{equation*}
Therefore $(h, a_{1,2}, \dots, a_{k,m_k})$ is trivial.
\end{proof}

We may now use the preceding lemmas to prove Theorem \ref{thm:OneHyp}.
\begin{proof}By Lemma \ref{lem:gens}, at most $1 + n |V| + n \sum_{v \in V} (m_v -2)$ homologically distinct cycles generate $H_1(K,\KK)$. Lemma \ref{lem:rels} shows that each multiple point eliminates a disjoint set of generators. Thus, combined with Lemma \ref{lem:infinity}, we have 
\begin{align*}
b_1(F, \KK) &\leq b_1(K,\KK) \\
 &\leq 1 + n |V| + n \sum_{v \in V} (m_v -2) - 1 - \sum_{v \in V} \left[ (m_v -2)(n-\gcd(m_v , n)) + (n-1)\right] \\
 &= n |V| + n \sum_{v \in V} (m_v -2) - n \sum_{v \in V} (m_v -2) \\
 & \hspace*{0.95in}  + \sum_{v \in V} \left[ (m_v -2)\gcd(m_v,n) \right] - |V|(n-1)\\
 &= |V| + \sum_{v \in V} \left[ (m_v -2)\gcd(m_v,n) \right]
\end{align*}

Note that $(n-1) +|V| =   \sum_{v \in V} m_v$, thus we have
\begin{align*}
b_1(F, \KK)  &\leq |V| + \sum_{v \in V} \left[ (m_v -2)\gcd(m_v,n) \right] + |V| + (n-1) - \sum_{v \in V} m_v \\
 &= (n-1) + \sum_{v \in V} \left[ (m_v -2)\gcd(m_v,n) \right] +2 |V| - \sum_{v \in V} m_v \\
 &= (n-1) + \sum_{v \in V} \left[ (m_v -2)\gcd(m_v,n) \right] +2 \sum_{v \in V} 1 - \sum_{v \in V} m_v \\
 &= (n-1) + \sum_{v \in V} \left[ (m_v -2)(\gcd(m_v,n) -1) \right]
\end{align*}as desired in the statement of the theorem.\end{proof}

\subsection{One higher order multiple point} $~$

Finally, we examine the case where a line contains only one higher order multiple point.
  	\begin{theorem}
  	\label{thm:one-rel-prime}
  	Let $\arr$ be a complexified real arrangement in $\CP^2$ and let $F$ be the associated Milnor fiber. If there exists a line $H \in \arr$ such that there is only one multiple point $v \in H$ satisfying $\gcd(m(v), |\arr|) \neq 1$ and $m(v) > 2$, then 
  	\begin{equation*}
  	H_1(F,\ZZ) \cong \ZZ^{|\arr | - 1}.
  	\end{equation*}
  	\end{theorem}

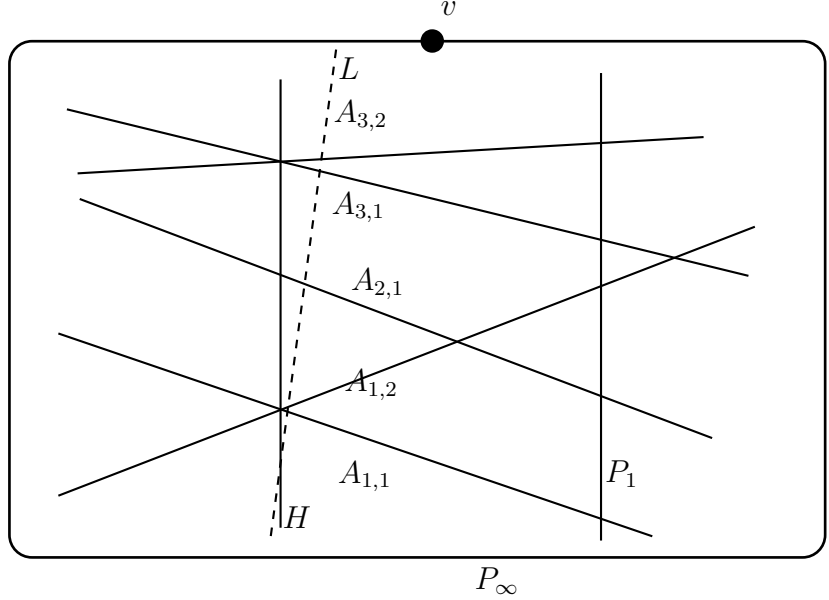
\begin{figure}[ht!]
	\begin{center}
	\input{one-point-intersection}
	\caption{An arrangement satisfying the conditions of Theorem \ref{thm:one-rel-prime}.}
	\label{fig:deconeOfCR-oneNotRelPrimePt}
\end{center}
\end{figure}

\begin{proof}
Let $P_\infty$ be any line in the arrangement containing the point $v$ except for $H$. Let $\darr$ be the decone of the arrangement with respect to the line $P_\infty$ (ie, the affine arrangement with $P_\infty$ the line at infinity). As we assumed that $\arr$ was a complexified real arrangement, we may depict the arrangement locally as in Figure \ref{fig:deconeOfCR-oneNotRelPrimePt}.

We will denote the lines intersecting $H$ by $\{A_{(l,k)} \}_{l=1,k=1}^{m,m_l}$ where the first index is with respect to each point of intersection and the second with respect to the lines in the intersection. Denote the lines parallel to $H$ by $\{P_j\}_{j=1}^{p}$. Further, we index the $P_j$ by increasing distance from $H$. (As they are parallel, the real parts of the lines have a well-defined distance between them. Therefore, we assume the distance from $H$ to $P_h$ is less than or equal to the distance from $H$ to $P_{h+1}$) . By considering a hyperplane slightly askew to $H$ (labeled $L$ in Figure \ref{fig:deconeOfCR-oneNotRelPrimePt} and denoted by a dashed line), we associate to each line a generator in the fundamental group coming from a small loop around the line oriented positively with respect to the natural complex structure and denoted by $h$, $a_{(l,k)}$ or $p_j$. Thus we construct an Arvola-Randell presentation of the form $\langle h, a_{(l,k)}, p_j : R \rangle$. 

Recall from the proof of Lemma \ref{lem:gens} that we may represent the generators of the first homology group of $K$ with respect to $h$ by
\begin{equation}
\label{eqn:gens2}
\left\lbrace
\begin{array}{l|r}
x^i . (a_{(l,k)} - h, 0, \dots, 0) & 0 \leq i \leq n-1 \\
x^i . (p_j - h, 0, \dots, 0) & 1 \leq l \leq m \\
(h, h, \dots, h) & 1 \leq k \leq m_l \\
 & 1 \leq j \leq p
\end{array}
\right\rbrace
\end{equation}
for tuples of length $n=|\arr|$. 

As the order of each multiple point along $H$ is relatively prime to $n$, we alter the relators in $H_1(K)$ arising from the multiple points along $H$ by Tietze transformations to have the form
\begin{equation}
\label{eqn:relators-aik}
\left\lbrace x^i . (a_{(l,k)} - h, -(a_{(l,k)}-h) , \dots, 0) \right\rbrace .
\end{equation}
This follows from Section \ref{ssec:one-point-intersection}.

As the lines $P_j$ are parallel to $H_0$ and we have at least one line $A_{(l,k)}$, there is a set of relations in $R$ of the form 
\begin{equation*}
[p_1, \alpha_1^{\Gamma_1}, \alpha_2^{\Gamma_2}, \dots, \alpha_{m_1}^{\Gamma_{m_1}}]
\end{equation*}
where each $\alpha_w$ is some $a_{(l,k)}$ and each $\Gamma_w$ is a word composed of some combination of $a_{(l,k)}$'s. 

Consider the set of relators in $H_1(F)$ generated by the relations
\begin{equation*}
\alpha_{m_1}^{\Gamma_{m_1}} \cdot p_1 \cdot \alpha_1^{\Gamma_1} \cdots \alpha_{m_1-1}^{\Gamma_{m_1-1}} =p_1 \cdot \alpha_1^{\Gamma_1} \cdots \alpha_{m_1}^{\Gamma_{m_1}} .
\end{equation*}
One may easily seen that the word $\alpha_{m_1}^{\Gamma_{m_1}}$ has one more positive exponent than negative exponent, thus will move the indexing one space up when the relator is lifted to $K$. As each subsequent $\alpha^\Gamma$ has the same form, the result will be that all the conjugations do not affect the starting point of each subsequent conjugated term in the relation. Thus, the relators will have the form $x^i. (h-p_1 + A_1, p_1 - h + A_2, A_3, \dots, A_{m_1}, 0, \dots, 0)$, where the $A_i$ are combinations of $a_{(j,k)}$'s. Every $a_{(i,k)}$ occurs an even number of times in the above relations (as they are involved in conjugation or on both sides of the equal sign), thus using the relators from \eqref{eqn:relators-aik} we may remove all $(a_{(i,k)} - h)$ terms from the relators. The end result is
 \begin{equation*}
x^i. (h-p_1, -(h- p_1) , 0, \dots, 0).
\end{equation*} 

Any other parallel line $P_g$ will generate relations of the form
\begin{equation*}
[p_g, \alpha_1^{\Gamma_1}, \alpha_2^{\Gamma_2}, \dots, \alpha_{m_g}^{\Gamma_{m_g}}]
\end{equation*}
where the $\alpha_i$ are single letters of the form $a_{(l,k)}$ or $p_{g'}$ where $g'<g$. The $\Gamma$ terms are words in $a_{(l,k)}$ or $p_{g'}$ where $g'<g$. By the same arguments as given above, we may conclude that we have relators of the form
 \begin{equation*}
 x^i. (h-p_g,-(h-p_g), 0, \dots, 0).
 \end{equation*}
 
 Therefore, we have a presentation for $H_1(K,\KK)$ given by 
 \begin{equation}
 \left(
 \begin{array}{l | l}
 x^i . (a_{(l,k)} - h, 0, \dots, 0) & x^i . (a_{(l,k)} - h, -(a_{(l,k)}-h) , \dots, 0)\\
 x^i . (p_j - h, 0, \dots, 0) & x^i. (h-p_j, -(h-p_j), 0, \dots, 0) \\ 
 (h, h, \dots, h) & R* 
 \end{array}
 \right)
 \end{equation}
 Therefore, the group has at most $|\arr| -1$ generators, so it follows that must be a free abelian group on $|\arr| -1$ generators as we know that $b_1(F,\KK) = b_1(K,\KK) \geq |\arr| - 1$.
\end{proof}

\section{Examples}\label{sec:example}

Cohen, Dimca and Orlik prove the following theorem in \cite{CDO-NonresonanceConditions-MR2038782}:

\begin{theorem}\label{thm:CDO-LocSystems}
 Let $\arr$ be an arrangement of $n$ projective lines in $\CP^2$, with associated Milnor fiber $F$. Then for any integer $0<k<n$ and any line $H$ in the arrangement $\arr$ we have
\begin{equation*}
b_1(F, \CC)_k \leq \sum_x (m_x -2)
\end{equation*}
where the sum is over all points $x \in H$ such that the multiplicity of $\arr$ at $x$ is $m_x > 2$ and $n$ divides $k m_x$.
\end{theorem}

If one is only interested in computing the complex betti number of the Milnor fiber, then Example \ref{ex:cdo-stronger} shows that Theorem \ref{thm:CDO-LocSystems} will allow one to conclude that the betti number is minimal in cases where Theorem \ref{thm:OneHyp} does not.

\begin{example}\label{ex:cdo-stronger}Let $\arr$ be an arrangement of 12 projective lines in $\CP^2$ such that there is a line $H_3 \in \arr$ that contains multiple points of order 2 and 3 only, and contains at least one point of order 3. We also assume that $H_4 \in \arr$ is a line that contains points of order 2 and 4 only and contains at least one point of order 4. By Theorem \ref{thm:OneHyp}, $b_1(F,\KK) \leq 11 + 2m_3$ where $m_3$ is the number of multiple points of order 3 contained in $H_3$. 

However, applying Theorem \ref{thm:CDO-LocSystems} to the arrangement for $k=1,2,3,5,6,7,9,10,11$ along the line $H_3$ to conclude that $b_1(F)_k = 0$. Applying Theorem \ref{thm:CDO-LocSystems} for $j=1,2,4,5,7,8,10,11$  along $H_4$ allows us to conclude that $b_1(F)_j =0$. Therefore, $b_1(F, \CC) = 11$.
\end{example}

\begin{example} Let $\arr$ be an arrangement of lines in $\CP^2$ such that one line only contains multiple points that are relatively prime to the order of the arrangement or have order two. Then Theorem \ref{thm:CDO-LocSystems} and Theorem \ref{thm:OneHyp} both imply that the associated Milnor fiber has minimal betti number. However, by Theorem \ref{thm:OneHyp} we may also conclude that the first homology group is torsion free.\end{example}

\begin{example}\label{ex:torsion-free}Let $\arr$ be an arrangement in $\CP^2$ such that the number of lines in the arrangement is $p^d$ for some prime $p$ greater than two. Also, suppose that all lines in the arrangement have at least one multiple point of order $p$, and at least one line $H$ has only one multiple point with order divisible by $p$. In this case, Theorem \ref{thm:CDO-LocSystems} for any line will at best yield the inequality $b_1(F,\CC)_k \leq (p-2)$. However, applying Theorem \ref{thm:one-rel-prime} to the line $H$ will yield $H_1(F;\ZZ) \cong \ZZ^{p^d -1}$; hence minimal first betti number and torsion free first homology group.
\end{example}

\providecommand{\bysame}{\leavevmode\hbox to3em{\hrulefill}\thinspace}
\providecommand{\MR}{\relax\ifhmode\unskip\space\fi MR }
\providecommand{\MRhref}[2]{%
  \href{http://www.ams.org/mathscinet-getitem?mr=#1}{#2}
}
\providecommand{\href}[2]{#2}

 \makeatletter 
\providecommand\@dotsep{5}
\makeatother

\end{document}

%% file: local-multiple-point-graph.tex
\begin{tikzpicture}[y=0.80pt, x=0.8pt,yscale=-1, inner sep=0pt, outer sep=0pt,scale=0.5]
  \begin{scope}[shift={(-91.739216,-394.62531)}]
    \path[draw=black,line join=miter,line cap=butt,line width=0.800pt]
      (467.8234,595.9331) -- (306.0897,978.3565);
    \path[draw=black,line join=miter,line cap=butt,line width=0.800pt]
      (233.6108,940.8680) -- (544.6500,626.8999);
    \path[draw=black,line join=miter,line cap=butt,line width=0.758pt]
      (252.9793,611.6631) -- (529.7947,976.6918);
    \path[draw=black,line join=miter,line cap=butt,line width=1.022pt]
      (206.1932,695.9605) -- (589.7011,890.6813);
    \path[draw=black,line join=miter,line cap=butt,line width=0.893pt]
      (196.3302,755.1890) -- (608.7323,826.7209);
    \path[fill=black] (538.7113,995.00726) node[above right] (text3667) {$a_1
      $};
    \path[fill=black] (601.9552,899.18323) node[above right] (text3671) {$a_2
      $};
    \path[shift={(0,552.36218)},fill=black] (532.6087,45.652176) node[above right]
      (text3675) {$a_{m-1}       $};
    \path[fill=black] (628.30768,829.11548) node[above right] (text3679) {$a_3
      $};
    \path[fill=black] (503.18408,751.3454) node[above right] (text3683) {$...
      $};
    \path[fill=black] (425.46674,559.28961) node[above right] (text3675-5) {$a_{m}
      $};
    \path[fill=black] (210.93723,591.38824) node[above right] (text2970) {$a'_1
      $};
    \path[shift={(121.03364,415.83852)},fill=black] (86.873116,366.4686) node[above
      right] (text2974) {       };
    \path[fill=black] (153.69006,695.29108) node[above right] (text2970-3) {$a'_2
      $};
    \path[fill=black] (141.56824,779.13373) node[above right] (text2970-0) {$a'_3
      $};
    \path[fill=black] (122.37534,945.80896) node[above right] (text2970-06)
      {$a'_{m-1}     $};
  \end{scope}
  \path[fill=black] (113.46863,420.87387) node[above right] (text2970-6) {...
    };
  \path[fill=black] (159.93565,604.72168) node[above right] (text2970-9) {$a'_m$
    };

\end{tikzpicture}

%% file: one-line-intersection.tex
\begin{tikzpicture}[y=0.80pt,x=0.80pt,yscale=-1, inner sep=0pt, outer sep=0pt,scale=0.5]
\begin{scope}[shift={(0,-52.362183)}]
  \path[shift={(0,52.362183)},draw=black,line join=miter,line cap=butt,line
    width=0.800pt] (454.5686,244.4059) .. controls (454.5686,668.6700) and
    (454.5686,668.6700) .. (454.5686,668.6700);
  \path[shift={(0,52.362183)},draw=black,line join=miter,line cap=butt,line
    width=0.800pt] (244.4569,638.3654) -- (903.0764,383.8069);
  \path[shift={(0,52.362183)},draw=black,line join=miter,line cap=butt,line
    width=0.800pt] (244.4569,484.8222) -- (806.1017,676.7512);
  \path[shift={(0,52.362183)},draw=black,line join=miter,line cap=butt,line
    width=0.800pt] (264.6600,357.5430) -- (862.6703,583.8171);
  \path[shift={(0,52.362183)},draw=black,line join=miter,line cap=butt,line
    width=0.800pt] (262.6397,333.2993) -- (854.5891,298.9541);
  \path[shift={(0,52.362183)},draw=black,line join=miter,line cap=butt,line
    width=0.800pt] (252.5381,272.6902) -- (897.0155,430.2740);
  \path[fill=black] (456.58896,721.0321) node[above right] (text3703) {$H$     };
  \path[fill=black] (509.11688,686.68701) node[above right] (text3711) {$A_{1,2}$
    };
  \path[fill=black] (515.1778,599.81384) node[above right] (text3715) {$A_{1,3}$
    };
  \path[fill=black] (521.23871,504.85953) node[above right] (text3719) {$A_{2,2}$
    };
  \path[shift={(0,52.362183)},fill=black] (503.05597,379.76633) node[above right]
    (text3723) {$A_{3,2}$     };
  \path[shift={(0,52.362183)},fill=black] (505.07626,292.89322) node[above right]
    (text3727) {$A_{3,3}$     };
  \path[shift={(0,52.362183)},miter limit=4.00,line width=0.080pt]
    (897.0155,485.8324) .. controls (897.0155,650.4105) and (726.9655,783.8274) ..
    (517.1981,783.8274) .. controls (307.4308,783.8274) and (137.3808,650.4105) ..
    (137.3808,485.8324) .. controls (137.3808,321.2543) and (307.4308,187.8374) ..
    (517.1981,187.8374) .. controls (726.9655,187.8374) and (897.0155,321.2543) ..
    (897.0155,485.8324) -- cycle;
\end{scope}

\end{tikzpicture}

%% file: one-point-intersection.tex
\begin{tikzpicture}[y=0.80pt, x=0.8pt,yscale=-1, inner sep=0pt, outer sep=0pt,scale=0.5]
\begin{scope}[shift={(0,-52.362183)}]
  \path[shift={(0,52.362183)},draw=black,line join=miter,line cap=butt,line
    width=0.800pt] (454.5686,244.4059) .. controls (454.5686,668.6700) and
    (454.5686,668.6700) .. (454.5686,668.6700);
  \path[shift={(0,52.362183)},draw=black,line join=miter,line cap=butt,line
    width=0.800pt] (757.6144,238.3450) -- (757.6144,680.7918);
  \path[shift={(0,52.362183)},draw=black,line join=miter,line cap=butt,line
    width=0.800pt] (244.4569,638.3654) -- (903.0764,383.8069);
  \path[shift={(0,52.362183)},draw=black,line join=miter,line cap=butt,line
    width=0.800pt] (244.4569,484.8222) -- (806.1017,676.7512);
  \path[shift={(0,52.362183)},draw=black,line join=miter,line cap=butt,line
    width=0.800pt] (264.6600,357.5430) -- (862.6703,583.8171);
  \path[shift={(0,52.362183)},draw=black,line join=miter,line cap=butt,line
    width=0.800pt] (262.6397,333.2993) -- (854.5891,298.9541);
  \path[shift={(0,52.362183)},draw=black,line join=miter,line cap=butt,line
    width=0.800pt] (252.5381,272.6902) -- (897.0155,430.2740);
  \path[fill=black] (456.58896,721.0321) node[above right] (text3703) {$H$     };
  \path[fill=black] (761.65497,682.64636) node[above right] (text3707) {$P_1$
    };
  \path[fill=black] (509.11688,686.68701) node[above right] (text3711) {$A_{1,1}$
    };
  \path[fill=black] (515.1778,599.81384) node[above right] (text3715) {$A_{1,2}$
    };
  \path[fill=black] (521.23871,504.85953) node[above right] (text3719) {$A_{2,1}$
    };
  \path[shift={(0,52.362183)},fill=black] (503.05597,379.76633) node[above right]
    (text3723) {$A_{3,1}$     };
  \path[shift={(0,52.362183)},fill=black] (505.07626,292.89322) node[above right]
    (text3727) {$A_{3,2}$     };
  \path[shift={(0,52.362183)},miter limit=4.00,line width=0.080pt]
    (897.0155,485.8324)arc(0.000:180.000:379.817350 and
    297.995)arc(-180.000:0.000:379.817350 and 297.995) -- cycle;
  \path[draw=black,line join=round,line cap=round,miter limit=27.00,line
    width=1.010pt,rounded corners=0.3000cm] (198.1837,260.5964) rectangle
    (969.5527,749.1226);
  \path[shift={(10.101525,54.382488)},draw=black,fill=black,line join=bevel,line
    cap=rect,miter limit=16.60,line width=2.400pt]
    (595.9900,206.0201)arc(-0.022:180.022:8.081)arc(-180.022:0.022:8.081) --
    cycle;
  \path[fill=black] (606.09155,236.15894) node[above right] (text3739) {$v$     };
  \path[shift={(0,52.362183)},fill=black] (638.41638,731.29944) node[above right]
    (text3743) {$P_\infty$     };
  \path[draw=black,dash pattern=on 3.26pt off 3.26pt,line join=miter,line
    cap=butt,miter limit=4.00,line width=0.814pt] (507.0877,268.4927) --
    (444.4760,735.1654);
  \path[fill=black] (509.11688,296.76807) node[above right] (text3749) {$L$   };
\end{scope}

\end{tikzpicture}